\documentclass{gssr2e}
\newcommand{\al}{\alpha}
\newcommand{\R}{{\mathbb R}}
\newcommand{\abs}[1]{\left\vert#1\right\vert}
\newcommand{\norm}[1]{\left\lVert#1\right\rVert}
\newcommand{\set}[1]{\left\{#1\right\}}
\newcommand{\ex}[1]{\mathsf{E}\left[#1\right]}

\begin{document}

\title{Mixed fractional stochastic differential equations with jumps}


\author{Georgiy Shevchenko$^{\rm a}$$^{\ast}$\thanks{$^\ast$Corresponding author. Email: zhora@univ.kiev.ua
\vspace{6pt}}\\\vspace{6pt}  $^{\rm a}${\em{Department of Probability Theory, Statistics and Actuarial Mathematics,
Taras Shevchenko National University of Kyiv,
64 Volodymyrska, 01601 Kyiv, Ukraine}}}

\maketitle

\begin{abstract}
In this paper, we consider a stochastic differential equation driven by a fractional Brownian motion (fBm) and a Wiener process and having jumps.
We prove that this equation has a unique solution and show that all moments of the solution are finite.
\begin{classcode}
Primary 60G15; Secondary 60G22, 60H10, 60J65.
\end{classcode}
\end{abstract}

\section{Introduction}

The main object of this paper is a stochastic differential equation
 \begin{equation}\label{sde-jumps}
 \begin{gathered}
 X_t =X_0 +\int_0^t a(s,X_{s})ds+
\int_0^tb(s,X_{s})dW_s+\int_0^tc(s,X_{s})dB_s^H\\
 +\int_0^t\int_{\R} q(s,X_{s-},y)\nu(ds,dy),\quad
t\in[0,T],
 \end{gathered}
 \end{equation}
where $W$ is a standard Wiener process; $B^H$ is a fractional Brownian motion (fBm) with the Hurst parameter $H\in(\frac{1}{2},1)$;
$\nu$ is a Poisson measure with finite intensity measure. 

Such equation can be used to model processes with jumps on financial markets, where two principal random noises influence
the prices. One noise is coming from economical background and has a long-range dependence, which is modeled by the fBm. Another noise
is intrinsic to the stock exchange, where millions of  agents act independently and behave  irrationally sometimes; this is a white noise
and it is modeled by a Wiener process.

Although equation \eqref{sde-jumps} were not studied before, many authors considered some particular examples.

For a pure fractional stochastic equation without Wiener component and jumps
 \begin{equation}\label{sde-pure}
X_t =X_0 +\int_0^t a(s,X_{s})ds+\int_0^tc(s,X_{s})dB_s^H,
 \end{equation}
existence and uniqueness of a solution to such equation was proved first in  \cite{kleptsyna, kubiliuspure,Ruzmaikina, russovallois}.
In \cite{Nual} this result was proved under weaker assumptions on coefficients and integrability of the solution was established
for $H>3/4$. For a homogeneous drift-less equation (i.e.\ $a(s,x)=0$, $c(s,x)=c(x)$), the integrability was shown in \cite{HuNual} for all $H>1/2$.

Mixed stochastic differential equation without jumps
 \begin{equation}\label{sde-mixed}
 X_t =X_0 +\int_0^t a(s,X_{s})ds+
\int_0^tb(s,X_{s})dW_s+\int_0^tc(s,X_{s})dB_s^H
 \end{equation}
was first considered  in  \cite{kubilius}, where unique solvability was proved for time-independent coefficients and zero drift, i.e.\ $a=0$. Later, in \cite{Zah05}, existence of solution to \eqref{sde-mixed} was proved under less restrictive assumptions, but only locally, i.e.\ up to a random time. In \cite{guernual} global existence and uniqueness of solution to \eqref{sde-mixed} was established under the assumption that $W$ and $B^H$ are independent. 
The latter result was obtained in \cite{mbfbm-sde} and \cite{mbfbm-limit} without the independence assumptions, and it was also shown in \cite{mbfbm-limit} that all moments of the solution are finite for $H>3/4$. It is also worth mentioning that article \cite{russovallois} contains related results, which imply, in particular, that \eqref{sde-mixed} has unique solution for $b(t,x) = c(t,x) = b(x)$. 

Recently, equation \eqref{sde-jumps} without Brownian component (i.e.\ $b=0$) was considered in \cite{baima}, where the existence of solution is proved
under weaker conditions on regularity of $a$ (only H\"older continuity in the second variable) and without the assumptions $\Pi(\R)<\infty$, $H>1/2$.
On the other hand, in that paper $c(t,x)\equiv1$ and $q(t,x,y)$ is independent of $x$, which are much stronger assumptions than those of the present paper.

In this  paper we show that \eqref{sde-jumps} has a unique solution. The main result is existence of moments of the solution, which is
an important property for applications. The existence of moments is proved for all $H>1/2$ in inhomogeneous case, 
which is a novelty in comparison to the results
of \cite{HuNual,Nual,mbfbm-limit}. The paper is organized as follows. Section~\ref{sec:prelim} gives basic definitions.
In Section~\ref{sec:exuniq} we prove existence and uniqueness of the solution to \eqref{sde-jumps}. Section~\ref{sec:exmom} contains results on integrability of solutions to
\eqref{sde-jumps}--\eqref{sde-mixed}.

\section{Preliminaries}\label{sec:prelim}

Let $\bigl(\Omega, \mathcal{F}, \mathbb{F} = \{\mathcal{F}_t, t\ge 0\},
\mathsf{P}\bigr)$ be a complete filtered probability space satisfying the usual assumptions.  Let also $\{W_t, t\ge 0\}$ be an
$\mathbb F$-Wiener process and $\{B^H_t, t\ge 0\}$ be an $\mathbb{F}$-adapted fractional Brownian motion (fBm), i.e.\ a centered Gaussian
process with the covariance $\ex{B_t^HB_s^H}=\frac{1}{2}(s^{2H}+t^{2H}-|t-s|^{2H})$. Let also $\nu$ be an $\mathbb{F}$-adapted Poisson measure with intensity measure $\Pi$, i.e.\  $$\ex{\nu(dt,dy)}=\Pi(dy)dt.$$ 
We will assume that the intensity measure $\Pi$ is finite: \begin{equation}\label{pifinite}
\Pi(\R)<\infty. 
\end{equation}
It is well known that $B^H$ has a modification with almost surely continuous paths (even H\"older continuous of any order up to $H$), and further we will assume that it is continuous itself.

Now we define how we understand the integrals in \eqref{sde-jumps}.
The integral with respect to the Wiener process $W$ is the standard It\^o integral, and it is well defined as long as $\int_0^t b(s,X_{s})^2ds<\infty$ almost surely. The integral with respect to $\nu$ is defined as usual. Since $\Pi(\R)<\infty$, the process
\begin{equation}\label{Lt}
L_t = \int_0^t\int_{\R} x\,\nu(ds,dx)
\end{equation}
is well defined and it is a pure jump process, which almost surely has finite number of jumps; we can assume that it is cadlag. So the integral with respect to $\nu$ is just a finite sum
$$
\int_0^t\int_{\R} q(s,X_{s-},y)\nu(ds,dy) = \sum_{s\in[0,t]} q(s,X_{s-},\Delta L_s),
$$
where $\Delta L_s = L_s-L_{s-}$.

As for the integral with respect to the fBm $B^H$, we use the generalized
Lebesgue--Stieltjes integral (see \cite{Zah98a}). Its construction uses the fractional derivatives
\begin{gather*}
\big(D_{a+}^{\alpha}f\big)(x)=\frac{1}{\Gamma(1-\alpha)}\bigg(\frac{f(x)}{(x-a)^\alpha}+\alpha
\int_{a}^x\frac{f(x)-f(u)}{(x-u)^{1+\alpha}}du\bigg),\\
\big(D_{b-}^{1-\alpha}g\big)(x)=\frac{e^{-i\pi
\alpha}}{\Gamma(\alpha)}\bigg(\frac{g(x)}{(b-x)^{1-\alpha}}+(1-\alpha)
\int_{x}^b\frac{g(x)-g(u)}{(u-x)^{2-\alpha}}du\bigg).
\end{gather*}
Assuming that $D_{a+}^{\alpha}f\in L_1[a,b], \ D_{b-}^{1-\alpha}g_{b-}\in
L_\infty[a,b]$, where $g_{b-}(x) = g(x) - g(b)$, the generalized (fractional) Lebesgue-Stieltjes integral $\int_a^bf(x)dg(x)$ is defined as
\begin{equation}\label{integr}\int_a^bf(x)dg(x)=e^{i\pi\alpha}\int_a^b\big(D_{a+}^{\alpha}f\big)(x)\big(D_{b-}^{1-\alpha}g_{b-}\big)(x)dx.
\end{equation}

It follows from H\"older continuity of $B^H$ that $D_{b-}^{1-\al}B_{b-}^H(x)\in L_\infty[a,b]$ a.s. Then for a function
$f$ with $D_{a+}^{\al}f\in L_1[a,b]$ we can define integral with respect to $B^H$ through  \eqref{integr}:
\begin{equation}\label{integrfbm}
\int_a^b f(x)\,dB^H(x):=e^{i\pi\alpha}\int_a^b (D_{a+}^\al f)(x)(D_{b-}^{1-\al} B_{b-}^H)(x)\,dx.
\end{equation}

Note that in the case where $f$ is piecewise H\"older continuous with exponent $\gamma>1-H$, this integral is just a
limit of forward integral sums $ \sum_{k=1}^n f(t_k)\big(B^H_{t_{k+1}}-B^H_{t_k}\big)$.
(This fact is proved in \cite{Zah98a} for H\"older continuous functions $f$, but it is easily checked
that the proof works for piecewise H\"older continuous functions as well.) Hence, for such functions, all
usual properties of integral hold: linearity, additivity etc.

Throughout the paper, the symbol $C$ will denote a generic constant, whose
value is not significant and can change from one line to another. To emphasize its dependence on some parameters,
we will put them into subscripts.

\section{Existence and uniqueness of solution}\label{sec:exuniq}
In this and the following sections we impose the following assumptions on the coefficients
of \eqref{sde-jumps}:
\begin{enumerate}[H1.]
\item
The function $c$ is differentiable in $x$ and for all $x,y\in\R$, $t\in[0,T]$
\begin{gather*}
|a(t,x)|+|b(t,x)|+|c(t,x)|\leq C(1+|x|),\\
|\partial_x c(t,x)|\leq C.
\end{gather*}
\item The functions $a$, $b$ and $\partial_x c$ are Lipschitz continuous in $x$:
$$|a(t,x_1)-a(t,x_2)|+|b(t,x_1)-b(t,x_2)|+|\partial_x c(t,x_1)-\partial_x c(t,x_2)|\leq C|x_1-x_2|,$$
for all $x_1, x_2\in{\R}$, $t\in[0,T]$.
\item The functions $a$, $b$ and $\partial_x c$ are H\"older continuous in $t$: for some  $\beta\in(1-H,1)$ and for all $s,t\in[0,T]$, $x\in\R$
$$|a(s,x)-a(t,x)|+|b(s,x)-b(t,x)|+|c(s,x)-c(t,x)|+|\partial_x c(s,x)-\partial_x c(t,x)|\leq C|s-t|^\beta.$$
\end{enumerate}
We do not impose any assumptions on the function $q$ except joint measurability in all arguments.

We will say that a process $X$ is a solution to \eqref{sde-jumps} if it is cadlag and has the following properties:
\begin{itemize}
\item for some $\kappa>1-H$
$$
\abs{X_t-X_s}\le C\abs{t-x}^\kappa
$$
for all $t,s$ such that $\nu([t-s],\R) = 0$ (i.e.\ $X$ is H\"older continuous between jumps of the process $L$ defined in \eqref{Lt});
\item equation \eqref{sde-jumps} holds almost surely for all $t\in[0,T]$.
\end{itemize}
From \eqref{pifinite} it follows that, almost surely,  $L$ has finitely many jumps on $[0,T]$,
so a solution $X$ is piecewise H\"older continuous of order $\kappa>1-H$, consequently, the integral $\int_0^t c(s,X_s) dB_s^H$
is well defined for all $t\in[0,T]$. It also follows that $\int_0^T (\abs{a(s,X_s)} ds + b(s,X_s)^2) ds<\infty$ a.s., thus the  integrals $\int_0^t a(s,X_s) ds$ and $ \int_0^t b(s,X_s) dW_s$ are well defined too.

\begin{theorem}\label{thm_ex_sol}
Equation \eqref{sde-jumps} has a unique solution.
\end{theorem}
\begin{proof}
For $u\ge 0$, consider the following equation:
$$
Y_t = y + \int_0^t a(s+u,Y_s)ds + \int_0^t b(s+u,Y_s)dV_s + \int_0^t c(s+u,Y_s) dZ_s,
$$
where $V$ is a Wiener process, $Z$ is a process with almost surely H\"older continuous paths of order $\gamma>1/2$.
It was proved in \cite{mbfbm-limit} that such equation has a unique solution in the class of H\"older continuous processes
of order $\kappa>1-\gamma$. We denote this solution by $Y(u,t,y,V,Z)$.

Let $\tau_n$ be the moment of the $n$th jump of process $L$. Define a sequence of processes $X^n$ recursively as follows.
Let the process $X^0_t = Y(0,t,X_0,W,B^H)$ be the solution to \eqref{sde-mixed}. If for $n\ge 1$ the process $X^{n-1}$ is constructed, set
$W^n_s = W_{\tau_n+s} - W_{\tau_n}$, $Z^n_s = B^H_{\tau_n+s} -B^H_{\tau_n}$,
 $s\ge 0$, $X^n_{\tau_n} = X^{n-1}_{\tau_n} + q(\tau_n,X^{n-1}_{\tau_n},\Delta L_{\tau_n})$. On the stochastic basis $(\Omega, \mathcal F^n, (\mathcal F^n_t, t\ge 0), P)$ with $\mathcal F^n = \sigma\{W^n,Z^n\}$,
$\mathcal F^n = \sigma\{W^n_s,Z^n_s,s\in[0,t]\}$, the process $W^n$ is a Wiener process, and  $Z^n_s$ is H\"older continuous
of any order $\gamma<H$, hence we can define
$X^n_t = Y(\tau_n,t-\tau_n, X^n_{\tau_n}, W^n, Z^n)$, $t\ge\tau_n$.

Now we put
\begin{equation*}
X_t= \sum_{n\ge 0} X^n_t 1_{[\tau_n,\tau_{n+1})}(t)
\end{equation*}
and show that this process solves \eqref{sde-jumps}. Indeed, for $t\in[\tau_n,\tau_{n+1})$, $n\ge 0$, we have
\begin{gather*}
X_t = X_{\tau_n} + \int_0^{t-\tau_n} \big(a(s+\tau_n,X_{s+\tau_n}) ds + b(s+\tau_n,X_{s+\tau_n}) dW^n_s + c(s+\tau_n,X_{s+\tau_n}) dZ^n_s\big) \\
= X_{\tau_n} + \int_{\tau_n}^{t} \left(a(s,X_{s}) ds + b(s,X_{s}) dW^n_{s+\tau_n} + c(s,X_{s}) dZ^n_{s+\tau_n}\right)\\
= X_{\tau_n} + \int_{\tau_n}^{t} \left(a(s,X_{s}) ds + b(s,X_{s}) dW_{s} + c(s,X_{s}) dB^H_s\right).
\end{gather*}
Thus, for any $t\ge 0$ we can write
\begin{gather*}
X_t = X_0 + \int_{\tau_n}^{t} \left(a(s,X_{s}) ds + b(s,X_{s}) dW_{s} + c(s,X_{s}) dB^H_s\right) + \sum_{n:\tau_n\le t} \Delta X_{\tau_n} \\
= X_0 + \int_{\tau_n}^{t} \left(a(s,X_{s}) ds + b(s,X_{s}) dW_{s} + c(s,X_{s}) dB^H_s\right) + \sum_{n:\tau_n\le t} q(\tau_n,X_{\tau_n-}, \Delta L_{\tau_n} ) \\
= X_0 + \int_{\tau_n}^{t} \left(a(s,X_{s}) ds + b(s,X_{s}) dW_{s} + c(s,X_{s}) dB^H_s\right) \\+ \int_0^t \int_{\R}
q(s,X_{s-},y)\nu(ds,dy),
\end{gather*}
i.e.\ $X$ solves \eqref{sde-jumps}.

Uniqueness follows from a similar reasoning: from uniqueness of solution for \eqref{sde-mixed}
and the strong Markov property of $W$ we get, that for $t\in[\tau_n,\tau_{n+1})$, $n\ge 0$, the solution of \eqref{sde-jumps} satisfies $X_t = Y(\tau_n,t-\tau_n, X^n_{\tau_n}, W^n, Z^n)$. On the other hand,
$\Delta X_{\tau_n} = q(\tau_n,X_{\tau_n-},\Delta L_{\tau_n})$, hence any solution of \eqref{sde-jumps}
coincides with the one constructed above.
\end{proof}

\section{Existence of moments of the solution}\label{sec:exmom}

\subsection{Existence of moments for equation without jumps}

We start by making  pathwise estimates of the solution to equation \eqref{sde-mixed} without jumps. We fix some
$\alpha\in(1-H,1/2)$ and introduce
the following notation:
\begin{gather*}
\norm{f}_{t} = \int_0^t \abs{f(t)-f(s)}(t-s)^{-1-\alpha} ds,\\
\norm{f}_{\lambda,t} =\sup_{s\le t} e^{-\lambda s} \abs{f(s)},\quad
\norm{f}_{1,\lambda,t} =\sup_{s\le t} e^{-\lambda s} \norm{f}_s,\quad
\norm{f}_{\infty;t} = \norm{f}_{0,t}+\norm{f}_{1,0,t};\\
\norm{f}_{0;[s,t]} = \sup_{s\le u<v<t} \left(\frac{\abs{f(v)-f(u)}}{(v-u)^{1-\alpha}} + \int_u^v \frac{\abs{f(u)-f(z)}}{(z-u)^{2-\alpha}}dz\right).
\end{gather*}
Observe that it follows from \eqref{integrfbm} that
\begin{equation}\label{integrestimate}
\begin{gathered}
\abs{\int_a^b f(s)dB^H_s} \le C \norm{B^H}_{0;[a,b]}\\\times \int_a^b \left(\abs{f(s)}(s-a)^{-\alpha} + \int_a^s \abs{f(s)-f(u)}(s-u)^{-1-\alpha}du \right)ds.
\end{gathered}
\end{equation}

First we establish some pathwise estimates of the solution of \eqref{sde-mixed}.
\begin{lemma}\label{pathwiseestimatemixed}
For the solution $X$ of \eqref{sde-mixed}, the following estimate holds:
\begin{equation*}
\norm{X}_{\infty;T} \le C \exp\set{C\norm{B^H}_{0;[0,T]}^{1/(1-\alpha)}}\left(1+\norm{I_b}_{\infty;T}\right),
\end{equation*}
where $I_b(t) = \int_0^t b(s,X_s)dW_s$, $t\in[0,T]$.
\end{lemma}
\begin{remark}
A similar estimate (naturally, without $J_b$) was obtained for the pure fractional equation \eqref{sde-pure} in \cite{Nual}, but with
exponent $1/(1-2\alpha)$ instead of $1/(1-\alpha)$ here. In our proof we will use methods similar to those of \cite{Nual}, but we modify them
as follows. While in \cite{Nual}, the sum $\norm{f}_{\lambda,t} + \norm{f}_{1,\lambda,t}$ is estimated and  a version
of the Gronwall lemma is used, here we will estimate these terms separately and then use a kind of two-dimensional Gronwall lemma.
\end{remark}

\begin{proof}
For shortness, denote $\Lambda=\norm{B^H}_{0;[0,T]}\vee 1$, $J_b = \norm{I_b}_{\infty;T}$, $h(t,s) = (t-s)^{-1-\alpha}$.

We start by estimating $\abs{X_t}$:
\begin{gather*}
\abs{X_t}\le \abs{X_0} + \abs{I_a(t)} + \abs{I_b(t)} + \abs{I_c(t)},
\end{gather*}
where $I_b$ is as above, $I_a(t) = \int_0^t a(s,X_s) ds$, $I_c(t) = \int_0^t c(s,X_s) dB^H_s$. Estimate $\abs{I_b(t)}\le J_b$,
\begin{gather*}
\abs{I_a(t)} \le \int_0^t \abs{a(s,X_s)}ds \le C\int_0^t \left(1+\abs{X_s}\right)ds\le C\left(1+\int_0^t \abs{X_s}ds\right).
\end{gather*}
By \eqref{integrestimate},
\begin{gather*}
\abs{I_c(t)} \le C\Lambda\int_0^t\left(\abs{c(s,X_s)}s^{-\alpha} + \int_0^s \abs{c(s,X_s) - c(u,X_u)} h(s,u)du\right)ds\\
\le C\Lambda \int_0^t\left(\left(1+\abs{X_s}\right)s^{-\alpha} + \int_0^s \left(\abs{s-u}^{\beta} + \abs{X_s - X_u}\right) h(s,u)du\right)ds\\
\le C\Lambda\left(1+ \int_0^t\left(\abs{X_s}s^{-\alpha} + \norm{X}_s\right)ds \right).
\end{gather*}
Summing up, we have
\begin{gather*}
\abs{X_t}\le C\Lambda\left(1+ \int_0^t\left(\abs{X_s}s^{-\alpha} + \norm{X}_s\right)ds \right) + J_b,
\end{gather*}
whence
\begin{gather*}
\norm{X}_{\lambda,t}\le C\Lambda\left(1+ \sup_{s\le t}e^{-\lambda s}\int_0^s\left(\abs{X_u}u^{-\alpha} + \norm{X}_u\right)du \right) + J_b\\
\le C\Lambda\left(1+ \sup_{s\le t}\int_0^se^{\lambda(u-s)}\left(e^{-\lambda u}\abs{X_u}u^{-\alpha} + e^{-\lambda u}\norm{X}_u\right)du \right) + J_b\\
\le C\Lambda\left(1+ \sup_{s\le t}\int_0^s e^{\lambda(u-s)}\left(u^{-\alpha} \norm{X}_{\lambda,t}  + \norm{X}_{1,\lambda,t}\right)du \right) + J_b\\\le
C\Lambda\left(1+ \lambda^{\alpha-1}\norm{X}_{\lambda,t}  + \lambda^{-1}\norm{X}_{1,\lambda,t}\right) + J_b,
\end{gather*}
where we have used the estimate
\begin{gather*}
\sup_{s\le t}\int_0^s e^{\lambda(u-s)}u^{-\alpha} du =\sup_{s\le t}\lambda^{-1}\int_0^{\lambda s} e^{-z}( s - z/\lambda)^{-\alpha} dz\\= \sup_{s\le t}\lambda^{\alpha-1}\int_0^{\lambda s} e^{-z}(\lambda s - z)^{-\alpha} dz\le  \lambda^{\alpha-1}\sup_{a>0}\int_0^{a} e^{-z}(a - z)^{-\alpha} dz =C\lambda^{\alpha-1}.
\end{gather*}
Further, we estimate $\norm{X}_t$:
\begin{gather*}
\norm{X}_t\le \norm{I_a}_t + \norm{I_b}_t + \norm{I_c}_t\le \norm{I_a}_t + J_b + \norm{I_c}_t,\\
\norm{I_a}_t \le \int_0^t \int_s^t\abs{a(u,X_u)}du\, h(t,s) ds \le C\int_0^t \int_s^t\left(1+\abs{X_u}\right)du\, h(t,s) ds\\
\le C\left(1+ \int_0^t \abs{X_u} (t-u)^{-\alpha}du\right),\\
\norm{I_c}_t = \int_0^t \abs{\int_s^t c(u,X_u) dB_u^H}h(t,s) ds
\le C\Lambda(J'_c+J''_c),
\end{gather*}
where
\begin{gather*}
J'_c = \int_0^t \int_s^t\abs{c(u,X_u)}(u-s)^{-\alpha}du\,h(t,s)ds\le  \int_0^t \int_s^t\left(1+\abs{X_u}\right)(u-s)^{-\alpha}du\, h(t,s)ds\\
\le C\left(1+\int_0^t\abs{X_u} \int_0^u(u-s)^{-\alpha}(t-s)^{-1-\alpha}ds \right)\le C\int_0^t \abs{X_u} (t-u)^{-2\alpha} du,
\\
J''_c = \int_0^t \int_s^t \int_s^u \abs{c(u,X_u)-c(z,X_z)}h(u,z)dz\,du\, h(t,s) ds \\
\le C\int_0^t \int_s^t \int_s^u \left((u-z)^{\beta} + \abs{X_u-X_z}h(u,z)\right)dz\,du\, h(t,s) ds\\
\le C\left(\int_0^t \int_s^t (u-s)^{\beta-\alpha} + \int_0^t \int_0^u \abs{X_u-X_z}h(u,z) (t-z)^{-\alpha} dz\,du\right)\\
\le C\left(1 + \int_0^t \int_0^u \abs{X_u-X_z}h(u,z)  dz(t-u)^{-\alpha} du\right) = C\left(1 + \int_0^t\norm{X}_u(t-u)^{-\alpha} du\right).
\end{gather*}
Here to estimate $J_c'$ we used the following computation:
\begin{gather*}
\int_0^u (u-s)^{-\alpha} (t-s)^{-1-\alpha} du  = \Big|s = u - (t-u)v \Big| = (t-u)^{-2\alpha} \int_0^{\frac{u}{t-u}} v^{-\alpha} (1+v)^{-1-\alpha} dt \\\le
(t-u)^{-2\alpha} \int_0^{\infty} v^{-\alpha} (1+v)^{-1-\alpha} dv = \mathrm{B}(1-\alpha,2\alpha) (t-u)^{-2\alpha}.
\end{gather*}
Combining all estimates, we get
\begin{gather*}
\norm{X}_t\le C\Lambda\left(1 + \int_0^t \left(\abs{X_u} (t-u)^{-2\alpha}+ \norm{X}_u(t-u)^{-\alpha} \right)du \right) + J_b,
\end{gather*}
whence
\begin{gather*}
\norm{X}_{1,\lambda,t}\le C\Lambda\left(1+ \sup_{s\le t}e^{-\lambda s}\int_0^s\left(\abs{X_u}(s-u)^{-2\alpha} + \norm{X}_u(s-u)^{-\alpha}\right)du \right) + J_b\\
\le C\Lambda\left(1+ \sup_{s\le t}\int_0^s e^{\lambda(u-s)}\left(e^{-\lambda u}\abs{X_u}(s-u)^{-2\alpha} + e^{-\lambda u}\norm{X}_u(s-u)^{-\alpha}\right)du \right) + J_b\\
\le C\Lambda\left(1+ \sup_{s\le t}\int_0^s e^{\lambda(u-s)}\left( \norm{X}_{\lambda,t}(s-u)^{-2\alpha}  + \norm{X}_{1,\lambda,t}(s-u)^{-\alpha}\right)du \right) + J_b\\\le
C\Lambda\left(1+ \lambda^{2\alpha-1}\norm{X}_{\lambda,t}  + \lambda^{\alpha-1}\norm{X}_{1,\lambda,t}\right) + J_b.
\end{gather*}
Thus, we get the following system of inequalities :
\begin{gather*}
\norm{X}_{\lambda,t}\le K\Lambda\left(1+ \lambda^{\alpha-1}\norm{X}_{\lambda,t}  + \lambda^{-1}\norm{X}_{1,\lambda,s}\right) + J_b,\\
\norm{X}_{1,\lambda,t}\le K\Lambda\left(1+ \lambda^{2\alpha-1}\norm{X}_{\lambda,t}  + \lambda^{\alpha-1}\norm{X}_{1,\lambda,s}\right) + J_b
\end{gather*}
with some constant $K$ (which can be assumed to be greater than 1 without loss of generality). Putting $\lambda = (4K \Lambda)^{1/(1-\alpha)}$, we get from the first inequality that
\begin{equation}\label{xlambdat}
 \norm{X}_{\lambda,t} \le \frac43 K\Lambda\left(1  + \lambda^{-1}\norm{X}_{1,\lambda,t}\right) + \frac43J_b.
\end{equation}
We remark that $\norm{X}_{\lambda,t}$ and $\norm{X}_{1,\lambda,t}$ are almost surely finite by the results of \cite{mbfbm-sde}.

Plugging this to the second inequality and making  simple transformations, we arrive at
\begin{gather*}
 \norm{X}_{1,\lambda,t} \le \frac32 K\Lambda + 2 K\Lambda^{1/(1-\alpha)} + 2K\Lambda J_b \le C \Lambda^{1/(1-\alpha)}(1 + J_b)
\end{gather*}
with some constant, which is no longer of interest.
Substituting this to \eqref{xlambdat}, we get
$$
 \norm{X}_{\lambda,t} \le C\Lambda (1+J_b)\le C\Lambda^{1/(1-\alpha)} (1+J_b).
$$
Finally,
\begin{gather*}
\norm{X}_{\infty;T} \le e^{\lambda T} \left( \norm{X}_{\lambda,t} + \norm{X}_{1,\lambda,t}\right) \le C\exp\set{C\Lambda^{1/(1-\alpha)}} \Lambda^{1/(1-\alpha)}(1+J_b)\\
\le C\exp\set{C\Lambda^{1/(1-\alpha)}}(1+J_b)\le  C\exp\set{C\left(1+\norm{B^H}_{0;[0,T]}^{1/(1-\alpha)}\right)}(1+J_b)\\\le C \exp\set{C\norm{B^H}_{0;[0,T]}^{1/(1-\alpha)}}(1+J_b),
\end{gather*}
as required.
\end{proof}

Now we are ready to state the result about finiteness of moments. To this end, in addition to our main hypotheses H1--H3, we will assume that the coefficient $b$ is bounded:
\begin{enumerate}[H1.]
\addtocounter{enumi}{3}
\item for all $x\in\R$, $t\in[0,T]$
$$
\abs{b(t,x)}\le C.
$$
\end{enumerate}
\begin{theorem}\label{moments-mixed}
The solution $X$ of \eqref{sde-mixed} for each $p > 0$ satisfies
\begin{equation*}
\ex{\norm{X}_{\infty;T}^p}<\infty,
\end{equation*}
in particular,
\begin{equation*}
\ex{\sup_{t\in[0,T]}\abs{X_t}^p}<\infty.
\end{equation*}
\end{theorem}
\begin{proof}
Thanks to Lemma~\ref{pathwiseestimatemixed}, it is enough to prove that all moments  of $\exp\set{\norm{B^H}^{1/(1-\alpha)}_{0;[0,T]}}$ and of $J_b$ are finite.  The
first follows from the fact that $\norm{B^H}_{0;[0,T]}$ is an almost surely finite supremum of a Gaussian family, and
$1/(1-\alpha)<2$, since $\alpha<1/2$.The first is proved as in \cite[Lemma 2.3]{mbfbm-limit}, but the proof is short and for completeness we repeat it here.

Denote $b_u = b(u,X_u)$ and write $$\ex{J_b^p}\le C_p(I_b' + I_b''),$$ where
\begin{gather*}
I_b' =\ex{\sup_{t\in[0,T]}\abs{\int_0^{t} b_s dW_s}^p }
\le C_p\ex{\left(\int_0^t \abs{b_s}^2ds\right)^{p/2}} <\infty,\\
I_b'' = \ex{\sup_{t\in[0,T]}\left(\int_0^t \abs{\int_s^t b_z dW_z} (t-s)^{-1-\alpha}ds\right)^p }.
\end{gather*}

By the Garsia--Rodemich--Rumsey inequality \cite[Theorem~1.4]{garsia}, for arbitrary $\eta\in(0,1/2-\alpha)$, $u,s\in[0,T]$
$$\abs{\int_u^s b_z dW_z}\le C \xi_\eta(T)\abs{s-u}^{1/2-\eta},
\quad \xi_\eta(t) = \left(\int_0^t \int_0^t \frac{\abs{\int_x^y b_v dW_v}^{2/\eta}} {\abs{x-y}^{1/\eta}}dx\,dy\right)^{\eta/2}.
$$
For $p\ge 2/\eta$  by the H\"older inequality
\begin{gather*}
\ex{\xi_\eta(t)^p}\le C_{p,\eta}\int_0^t \int_0^t \frac{\ex{\abs{\int_x^y b_v dW_v}^{p}}} {\abs{x-y}^{p/2}}dx\,dy\\
\le C_{p,\eta}\int_0^t \frac{\ex{\left(\int_x^y (b_v)^2 dv\right)^{p/2}}} {\abs{x-y}^{p/2}} dx\,dy<\infty,
\end{gather*}
whence
$$
I_b'' \le C \ex{\xi_\eta(T)^p}\sup_{t\in[0,T]}\left(\int_0^t (t-s)^{-1/2-\eta-\alpha}ds\right)^p <\infty,
$$
and the statement follows.
\end{proof}
As a corollary, we get a generalization of a result of \cite{Nual} where the existence of moments under linear growth of $c$ is proved only for $H>3/4$.
\begin{corollary}
The solution $X$ of \eqref{sde-pure} for each $p > 0$ satisfies
\begin{equation*}
\ex{\norm{X}_{\infty;T}^p}<\infty,
\end{equation*}
in particular,
\begin{equation*}
\ex{\sup_{t\in[0,T]}\abs{X_t}^p}<\infty.
\end{equation*}
\end{corollary}

\subsection{Existence of moments for the mixed equation with jumps}
Now turn to equation \eqref{sde-jumps}. To prove existence and uniqueness of its solution,  we did not make any assumptions about $\nu$ and $q$, except that $\nu$ has finite activity: $\Pi(\R)<\infty$.
To prove existence of moments, we will make further assumptions on the measure $\nu$ and the coefficient $q$ in  addition to H1--H4.

\begin{enumerate}[H1.]
\addtocounter{enumi}{4}
\item The measure $\nu$ is independent of $B^H$, $W$.
\item There exists a function $g\colon \R\to [0,\infty)$ such that for all $x\in\R$, $t\in[0,T]$
$$
\abs{q(t,x,y)}\le g(y)(1+\abs{x}).
$$
\item For all $p>0$
$$
\int_{\R} g(y)^p \Pi(dy)<\infty.
$$
\end{enumerate}

\begin{theorem}
The solution $X$ of \eqref{sde-jumps} for each $p > 0$ satisfies
\begin{equation*}
\ex{\sup_{t\in[0,T]}\abs{X_t}^p}<\infty.
\end{equation*}
\end{theorem}
\begin{proof}Let $\tau_n$ be the moment of the $n$th jump of process $L$.
As in the proof of  \ref{thm_ex_sol}, consider for $u\ge 0$ the equation
$$
Y_t = y + \int_0^t a(s+u,Y_s)ds + \int_0^t b(s+u,Y_s)dV_s + \int_0^t c(s+u,Y_s) dZ_s,
$$
where $V$ is a Wiener process, $Z$ is an adapted process with  almost surely $\gamma$-H\"older continuous paths for some $\gamma>1/2$, and denote the unique solution to this equation by $Y(u,t,y,V,Z)$.
A reasoning similar to that used in the proof of Lemma~\ref{pathwiseestimatemixed} gives
\begin{equation}
\label{solestimate}
\norm{Y(u,\cdot,y,V,X)}_{\alpha;[0,t]}\le  K\abs{y} \exp\set{K\norm{Z}_{0;[0,t]}^{1/(1-\alpha)}}(1+\norm{I_{b,Y}}_{\infty;t}),
\end{equation}
where $I_{b,Y}(t) = \int_0^t b(s+u,Y_s) dV_s$, the constant $K$ depends only on constants in assumptions H1--H4, without loss of generality we assume $K>1$.
Hence we get for $t<\tau_1$
\begin{gather*}
\abs{X_t} \le K\abs{X_0}\exp\set{K\norm{B^H}_{0;[0,\tau_1]}^{1/(1-\alpha)}}\left(1+ \norm{I_{b}}_{\infty;\tau_1}\right),
\end{gather*}
where $I_{b}(t) = \int_0^t b(s,X_s) dW_s$.

For convenience denote $\abs{x}_1=\abs{x}\vee 1$ and assume without loss of generality that $g(x)>1$, $x\in\R$.
Then
\begin{equation}\label{jumpestimate}
\sup_{t< \tau_1}\abs{X_t}_1 \le K\abs{X_0}_1\exp\set{K\norm{B^H}_{0;[0,\tau_1]}^{1/(1-\alpha)}} (1+J_{b}),
\end{equation}
where
\begin{gather*}
J_b = \sup_{0\le s\le t\le T} \int_s^t \left( \abs{\int_s^u b(v,X_v) dW_v}(u-s)^{-\alpha}\right.\\
 \left.+ \int_s^u \abs{\int_z^u b(v,X_v) dW_v} (u-z)^{-1-\alpha}dz \right)du.
\end{gather*}
Further,
\begin{gather*}
\abs{X_{\tau_1}}_1\le  \abs{X_{\tau_1-}}_1 + \abs{q(\tau_1,X_{\tau_1-},\Delta L_{\tau_1})}_1\\
\le \abs{X_{\tau_1-}}_1 + 2g(\Delta L_{\tau_1})\abs{X_{\tau_1-}}_1 \le  3g(\Delta L_{\tau_1})\abs{X_{\tau_1-}}_1\\
\le 3Kg(\Delta L_{\tau_1})\abs{X_0}_1\exp\set{K\norm{B^H}_{0;[0,\tau_1]}^{1/(1-\alpha)}} (1+J_{b}).
\end{gather*}
Using inequality \eqref{solestimate} consequently on $[\tau_1,\tau_2)$, $[\tau_2,\tau_3)$, $\dots$ and
estimating jumps of $X$ as in \eqref{jumpestimate}, we arrive at
\begin{gather*}
\sup_{t\in[0,T]}\abs{X_{t}}\le \sup_{t\in[0,T]}\abs{X_{t}}_1\le (3K)^{N(T)}(1+J_{b})^{N(T)+1}\abs{X_0}_1\exp\set{K\norm{B^H}_{0;[0,\tau_1\wedge T]}^{1/(1-\alpha)}}\\
\times \prod_{n:\tau_n\le T}g(\Delta L_{\tau_n})\exp\set{K\norm{B^H}_{0;[\tau_n,\tau_{n+1}\wedge T]}^{1/(1-\alpha)}}\\
\le C(6K)^{N(T)}\left(1+J_{b}^{N(T)+1}\right)\exp\set{K\norm{B^H}_{0;[0,\tau_1\wedge T]}^{1/(1-\alpha)}}\\
\times \prod_{n:\tau_n\le T}g(\Delta L_{\tau_n})\exp\set{K\norm{B^H}_{0;[\tau_n,\tau_{n+1}\wedge T]}^{1/(1-\alpha)}}
\end{gather*}
where $N(T)$ denotes the number of jumps of the process $L$ on $[0,T]$. Hence, by the H\"older inequality,
\begin{equation}\label{exsupxtp}
\begin{gathered}
\ex{\sup_{t\in[0,T]}\abs{X_{t}}^{p}}\le C_p\left(\vphantom{\ex{\prod_{n=1}^{N(T)} g(\Delta L_{\tau_n})^{4p}}}\ex{(6K)^{4pN(T)}}
\left(1+\ex{J_{b}^{4p(N(T)+1)}}\right)\right.\\
\times\left.\ex{\prod_{n=1}^{N(T)} g(\Delta L_{\tau_n})^{4p}}\ex{\prod_{n=0}^{N(T)}\exp\set{4Kp\norm{B^H}_{0;[\tau_n,\tau_{n+1}\wedge T]}^{1/(1-\alpha)}}}\right)^{1/4},
\end{gathered}
\end{equation}
with $\tau_0=0$. Since the jumps of $L$ are jointly independent and do not depend on $N(T)$, which has a Poisson distribution, we have
$\ex{(6K)^{4pN(T)}}<\infty$ and
\begin{gather*}
\ex{\prod_{n=1}^{N(T)} g(\Delta L_{\tau_n})^{4p}} = \exp\set{(\ex{g(\Delta L_{\tau_1})^{4p}}-1)\Pi(\R)T}\\
= \exp\set{\left(\int_{\R}g(y)^{4p}\Pi(dy)-1\right)\Pi(\R)T}<\infty.
\end{gather*}
Further, write
\begin{gather*}
\ex{J_{b}^{4p(N(T)+1)}} = \ex{\ex{J_{b}^{4p(N(T)+1)}\,\middle\vert\, L}}
\end{gather*}
From the formula for the solution to \eqref{sde-jumps}, obtained in the proof of Theorem~\ref{thm_ex_sol}, we get that
$X_t = F(t,W,B^H,L)$, where $F$ is certain non-random measurable function. Thus, we can write  $b(s,X_s) =
G(s,W,B^H,L)$, where $G$ is a non-random  bounded function. Therefore, since $W$ and $B^H$ do not depend on $L$, we obtain
\begin{gather*}
\ex{J_{b}^{4p(N(T)+1)}\,\middle\vert\, L} =
\ex{\sup_{0\le s\le t\le T} \norm{\int_s^\cdot G(u,W,B^H,l) dW_u}_{\alpha;[s,t]}^{4pk}}
\bigg|_{l=L, k= N(T)+1}.
\end{gather*}
Abbreviate  $G_u(l) = G(u,W,B^H,l)$ and estimate
\begin{gather*}
\norm{\int_s^\cdot G_u(l) dW_u}_{\alpha;[s,t]}\\\le \abs{\int_0^t G_u(l) dW_u} + \abs{\int_0^s G_u(l) dW_u}
+ \int_s^t \abs{\int_u^t G_z(l) dW_z}(t-u)^{-1-\alpha} du.
\end{gather*}
Then
\begin{gather*}
\ex{\sup_{0\le s\le t\le T} \norm{\int_s^\cdot G_u(l) dW_u}_{\alpha;[s,t]}^{4pk}}\le
2 (I_1 + I_2),
\end{gather*}
where
\begin{gather*}
I_1 = \ex{\sup_{t\in[0,T]} \abs{\int_0^t G_u(l) dW_u}^{4pk}} \le 
\ex{\left(\int_0^TG_u(l)^2 du\right)^{2pk}}\le K_1^k,\\
I_2 = \ex{\sup_{0\le s\le t\le T} \left(\int_s^t \abs{\int_u^t G_z(l) dW_z}(t-u)^{-1-\alpha} du\right)^{4pk}},
\end{gather*}
with some constant $K_1$ independent of $l,k$. The term $I_2$ is estimated as in the proof of Theorem~\ref{moments-mixed}:
from the Garsia--Rodemich--Rumsey inequality we get for any $\eta\in(0,1/2-\alpha)$, $u,s\in[0,T]$
$$\abs{\int_u^s G_z(l) dW_z}\le C_{\eta} \xi_\eta(T)\abs{s-u}^{1/2-\eta},$$
with
$$\xi_\eta(t) = \left(\int_0^t \int_0^t \frac{\abs{\int_x^y G_v(l) dW_v}^{2/\eta}} {\abs{x-y}^{1/\eta}}dx\,dy\right)^{\eta/2}.$$
Now for  $4pk \ge 2/\eta$
\begin{gather*}
\ex{\xi_\eta(T)^{4pk}}\le C_{\eta}^{4pk}T^{2 - 4pk\eta} \int_0^T \int_0^T \frac{\ex{\abs{\int_x^y G_v(l) dW_v}^{4pk}}} {\abs{x-y}^{2pk}}dx\,dy\\
\le C_{\eta}^{4pk}T^{2 - 4pk\eta} \int_0^t \frac{\ex{\left(\int_x^y G_v(l)^2 dv\right)^{2pk}}} {\abs{x-y}^{2pk}} dx\,dy
\le K_2^{k},
\end{gather*}
where the constant $K_2$ is independent of $l,k$. Consequently,
\begin{gather*}
I_2\le K_2^{k} \sup_{0\le s\le t\le T} \left(\int_s^t (t-u)^{-1/2-\alpha-\eta} du\right)^{4pk}\le (K_2 T^{4p(1/2-\alpha-\eta)})^{k}.
\end{gather*}
Collecting all estimates, we get
\begin{gather*}
\ex{\sup_{0\le s\le t\le T} \norm{\int_s^\cdot G_u(l) dW_u}_{\alpha;[s,t]}^{4pk}}\le
K_3^{k}
\end{gather*}
with $K_3$ independent of $l,k$. Therefore,
\begin{gather*}
\ex{J_{b}^{4p(N(T)+1)}} \le \ex{K_3^{N(T)+1}}<\infty.
\end{gather*}
Consider the last multiple in \eqref{exsupxtp}. Denote $Z_n = 4Kp\norm{B^H}_{0;[\tau_n,\tau_{n+1}\wedge T]}^{1/(1-\alpha)}$.
By the H\"older inequality,
\begin{gather*}
\ex{\prod_{n=0}^{N(T)}e^{Z_n}}\le \prod_{n=0}^{N(T)} \ex{e^{q_n Z_n}}^{1/q_n}= \prod_{n=0}^{N(T)}\ex{\ex{e^{q_n Z_n}|L}}^{1/q_n},
\end{gather*}
where
$$
q_n = \Delta_n^{-\kappa}\sum_{n=0}^{N(T)} \Delta_n^{\kappa},\quad \Delta_n = \tau_{n+1}\wedge T -\tau_n,\quad \kappa = \frac{\alpha+H-1}{1-\alpha}.
$$
Using the independence of  $B^H$ and $L$, write
$$
\ex{e^{q_n Z_n}|L} = \ex{\exp\set{4KpS (b-a)^{-\kappa}  \norm{B^H}_{0;[a,b]}^{1/(1-\alpha)}}}\Big\vert_{a= \tau_n,b=\tau_{n+1}\wedge T,S=\sum_{n=0}^{N(T)} \Delta_n^{\kappa}}.
$$
In view of the self-similarity property of $B^H$, it is easy to check that
$$
(b-a)^{-\kappa} \norm{B^H}_{0;[a,b]}^{1/(1-\alpha)}\overset{d}{=}\norm{B^H}_{0;[0,1]}^{1/(1-\alpha)}.
$$
Consequently,
\begin{gather*}
\ex{\prod_{n=0}^{N(T)}e^{Z_n}}\le \ex{\ex{\exp\set{4Kp S \norm{B^H}_{0;[0,1]}^{1/(1-\alpha)}}}
\Big\vert_{S=\sum_{n=0}^{N(T)} \Delta_n^{\kappa}}}
\end{gather*}
It is easy to check that $\kappa>1$. Therefore, $$\sum_{n=0}^{N(T)} \Delta_n^{\kappa}\le \left(\sum_{n=0}^{N(T)} \Delta_n\right)^{\kappa}=T^\kappa.$$
This implies
\begin{gather*}
\ex{\prod_{n=0}^{N(T)}e^{Z_n}}\le \ex{\exp\set{4Kp T^{\kappa} \norm{B^H}_{0;[0,1]}^{1/(1-\alpha)}}},
\end{gather*}
which is finite as $\norm{B^H}_{0;[0,1]}$ is a finite supremum of a Gaussian family and $1/(1-\alpha)<2$. The proof is now complete.
\end{proof}

\end{document}